\documentclass[12pt]{amsart}
\usepackage{latexsym,amscd,amssymb,epsfig}

\theoremstyle{plain}
  \newtheorem{theorem}{Theorem}[section]
  \newtheorem{proposition}[theorem]{Proposition}
  \newtheorem{lemma}[theorem]{Lemma}
  \newtheorem{corollary}[theorem]{Corollary}
  \newtheorem{conjecture}[theorem]{Conjecture}
\theoremstyle{definition}
  \newtheorem{definition}[theorem]{Definition}
  \newtheorem{example}[theorem]{Example}

 \theoremstyle{remark}

\numberwithin{equation}{section}

\def\QQ{{\mathbb Q}}

\def\FF{{\mathbb F}}

\def\complexes{{\mathbb C}}

\def\CC{{\mathcal C}}

\def\Sym{\mathfrak{S}}

\def\im{\mathrm{im}}
\def\unmatched{\mathrm{unM}}
\def\matched{\mathrm{M}}
\def\rank{\mathrm{rank}}
\def\mult{\mathrm{mult}}

\def\op{\mathrm{op}}
\def\Hom{{\mathrm{Hom}}}

\def\rev{{\mathsf r}{\mathsf e}{\mathsf v}}

\begin{document}

\title
{The $q=-1$ phenomenon 
via homology concentration}

\author{P. Hersh}
\address{North Carolina State University and Indiana University}
\email{phersh@indiana.edu}

\author{J. Shareshian}
\address{Washington University in St. Louis}
\email{shareshi@math.wustl.edu}

\author{D. Stanton}
\address{University of Minnesota}
\email{stanton@math.umn.edu}

\thanks{The authors were supported by NSF grants 
DMS-1200730, 
DMS-1202337, and DMS-1148634 respectively.  The first author was also 
supported by the Ruth Michler Prize of the Association for Women in Mathematics.}


\begin{abstract}
We introduce a homological approach to exhibiting instances of Stembridge's $q=-1$
phenomenon.  This approach is shown to explain two important instances of the
phenomenon, namely that of partitions whose Ferrers diagrams fit in a rectangle 
of fixed size and that of plane partitions fitting in a box of fixed size.
A more general framework of invariant and coinvariant complexes with 
coefficients taken mod 2 is developed, and as a part of this story 
an analogous homological result  for necklaces is conjectured.

\end{abstract}

\maketitle


\section{Introduction}
\label{intro}

There is a rich history surrounding the enumeration
of partitions in a rectangle or
higher dimensional box, as well as the enumeration of  
classes of partitions possessing various symmetries (see e.g. 
\cite{St-plane}, \cite{Ciu}).  
One reason for so much interest comes from connections to physics, 
while another is the important role they play in representation theory,
specifically in the theory of canonical bases (see e.g. \cite{St1} and 
\cite{St}).  
Richard Stanley used the Littlewood-Richardson rule 
in \cite{St-plane} to prove a recursive formula for the number of 
self-complementary plane partitions of bounded value.
John Stembridge proved that semistandard domino tableaux are counted
by this same formula, by showing that their 
enumeration formula satisfies the
same recurrence.  This proved that the set of 
fixed points in a fundamental involution
of Lusztig on a type A canonical basis also has this same
cardinality, by virtue of a bijection due to Berenstein
and Zelevinsky between the elements of a
canonical basis and semistandard Young tableaux (actually Gelfand-Tsetlin
patterns, objects also in bijection with semistandard Young tableaux) such that this
bijection sends Lusztig's involution to evacuation.
Stembridge examined this connection between self-complementary 
partitions and canonical bases more closely, 
unveiling in the process a phenomenon he dubbed
the ``$q=-1$ phenomenon''.  


In \cite{St}, Stembridge defines the $q=-1$ phenomenon as follows.  Given a finite set $B$ and a weight function $w:B \rightarrow {\mathbb N}_0$, one defines the generating function
\[
X(q):=\sum_{b \in B}q^{w(b)}.
\]
The $q=-1$ phenomenon occurs when there is some ``natural" involution $t$ on $B$ such that $X(-1)$ is the number of fixed points of $t$.  Stembridge exhibited various instances of the $q=-1$ phenomenon in which $B$ is an interesting collection of combinatorial objects with a natural 
weight function $w$, by showing in each case that $X(-1)$ is the trace of a matrix that is conjugate to the permutation matrix for $t$.

Here we introduce another method of proving that the $q=-1$ phenomenon occurs.  It is natural to ask if the phenomenon may be explained in interesting cases by an Euler characteristic computation.  That is, given $B,t$ as above, we seek a chain complex ${\mathcal C}$ such that, for each $i$,  the rank of the $i^{th}$ chain group $C_i$ is the coefficient of $q^i$ in $X(q)$.  Under these conditions,  the Euler characteristic of ${\mathcal C}$ is $X(-1)$.  We aim to use the fact that the Euler characteristic is also the alternating sum of the ranks of the homology groups and to choose such a complex  ${\mathcal C}$ so that its homology $H_\ast({\mathcal C})$ is concentrated in even dimensions and has a basis indexed by the set of fixed points of $t$, or at least by a set admitting a nice bijection to the fixed point set.  Thus the $q=-1$ phenomenon may be proven using the Euler-Poincar\'e formula. This approach amounts to a categorification of the $q=-1$ phenomenon.

In order to apply our method, we need some way of computing the homology of $\CC$.  In Section \ref{morse-sec}, we describe a rudimentary algebraic version of discrete Morse theory that suffices in the cases we examine herein.  In these cases, the chain groups $C_i$ are indexed by elements of $B$.  To compute $X(-1)$, one might find a fixed-point-free involution $m$ on some large subset $S$ of $B$ such that $|w(m(s))-w(s)|=1$ for all $s \in S$.  Then $X(-1)=|B \setminus S|$.  The idea behind algebraic discrete Morse theory is that, under the right conditions, $m$ can be used to reduce $\CC$ to a smaller complex $\CC^\prime$ such that $\CC$ and $\CC^\prime$ have the same homology and $H_\ast(\CC^\prime)$ is very easy to compute. 

We will carry out our homological approach in two quite central cases, namely when $B$ is the set of partitions whose Ferrers diagram fits in a rectangle of fixed size and when $B$ is the set of 3-dimensional partitions that fit in a box of fixed size.   In each case, the involution $t$ sends a partition to its complement in the given rectangle or box.  Thus, the $t$-fixed partitions are exactly those that are self-complementary. 

In the first case, our complex arises as a special case of the following general construction, which is developed in Section \ref{chain-complex-section}.  For a positive integer $n$, let $t$ be the involution on the power set $2^{[n]}$ mapping $Y$ to its complement $[n] \setminus Y$.   For any subgroup $G$ of the symmetric group $\Sym_n$, the action of $G$ on $[n]$ determines an action on $2^{[n]}$, and $t$ commutes with each element of $G$ in this action.  Thus $t$ acts on the set ${\mathcal O}(G)$ of orbits of $G$.   We describe in Section \ref{chain-complex-section} four chain complexes, each of which has Euler characteristic equal to the number of fixed points of $t$ on ${\mathcal O}(G)$.  Moreover, in each such complex, the rank of the $i^{th}$ chain space is the number of orbits of $G$ on the set ${{[n]} \choose {i}}$ of $i$-subsets of $[n]$.  Thus if we define the weight function $w$ on ${\mathcal O}(G)$ by mapping an orbit to the size of any representative, the resulting generating function $X(q)$ will give the number of orbits fixed by $t$ upon substituting $-1$ for $q$.  Thus we hope that the homology of one of these complexes is concentrated in even dimensions and has a basis indexed by the set of self -complementary $G$-orbits.  Our hope is realized when $G$ has an orbit of odd size on $[n]$, in which case the complexes are acyclic.  In particular, the complexes are acyclic when $|G|$ is odd.  We know also that each of these complexes has trivial first homology group.

In Section \ref{rect}, we indeed use one of the complexes from Section \ref{chain-complex-section} with $G$ a wreath product of symmetric groups with the set $B$ being the Ferrers diagrams fitting in a  rectangle of fixed size.  To this end, we make use of the following idea.
If we factor $n$ as $n=k\ell $ and identify $[n]$ with the set of boxes in a $k \times \ell$  rectangle.  We consider the partition of this set of boxes into $k$ subsets of size $l$, each subset consisting of all the boxes in a fixed row of the rectangle.  The stabilizer $G$ of this partition in $\Sym_n$ is isomorphic with the wreath product $\Sym_\ell \wr \Sym_k$.  Each $G$-orbit on $2^{[n]}$ contains the Ferrers diagram of a unique partition.   We show that indeed the homology of one of the associated complexes is concentrated in even dimensions and has a basis indexed by objects that are easily seen to be in bijection with the set of self-complementary partitions in our rectangle.

In Section \ref{plane} we turn to the case of partitions in a box.  We do not use the general construction from Section \ref{chain-complex-section}.  Rather, we use the fact that the set of plane partitions fitting in an $r \times c \times t$ box is equinumerous with the set of semi-standard Young tableaux of shape $(c^r)$ with entries from $\{0,1,\ldots,r+t-1\}$.  Moreover, there is a bijection between these two sets that intertwines the complementation map on the set of plane partitions and the evacuation map on the set of semi-standard Young tableaux.  We define a chain complex whose $i^{th}$ chain space has basis indexed by the set of semi-standard Young tableaux under consideration whose entries sum to $i+c{{r} \choose {2}}$.  We use the results of Section \ref{morse-sec} to show that the homology of this complex is concentrated in even dimensions and has a basis indexed by semi-standard domino tableaux of suitable bounded value.  Thus, the objects which previously arose in Stembridge's work also emerge in a completely natural way from our complexes and our Morse theory.   Section \ref{plane} concludes with a poset-theoretic result regarding posets playing an important role in our application  of algebraic Morse theory to the chain complexes studied therein.  We give a decomposition of each of these posets  into a disjoint union of Boolean algebras, which yields a formula for a principal specialization of the Schur function associated to the partition $(c^r)$.

In Section \ref{not-even}, we examine further the orbit complexes introduced in Section \ref{chain-complex-section}.  We describe an infinite family of permutation groups $G$ for which the homology of the associated complexes is not concentrated in even dimensions.  In fact, such concentration fails to occur whenever $n$ is even and a Sylow $2$-subgroup of $G$ contains no derangement.  Finally, we conjecture that this homology is concentrated in even dimensions when $G \leq \Sym_n$ is generated by an $n$-cycle.  In fact, we give a conjectured generating function for the the ranks of the homology groups in this case.

\section{Chain complexes resulting from a  group action}\label{chain-complex-section}


In this section, we define, for each permutation group on a finite set, an associated chain complex.  Certain complexes of this type will be used later to give a homological proof of the $q=-1$ phenomenon in the case of partitions in a rectangle.

Let $G$ be a subgroup of $\Sym_n$.  So, $G$ acts on $[n]:=\{1,2,\ldots,n\}$ and this action determines an action of $G$ on the set ${{[n]} \choose {i}}$ of subsets of cardinality $i$ from $[n]$ for each $i$.  Thus we have an action of $G$ on the set $2^{[n]}$ of all subsets of $[n]$.  For $S \subseteq [n]$, let $\overline{S}$ be the $G$-orbit containing $S$ in this action, and let $2^{[n]}/G$ be the set of all $G$-orbits on $2^{[n]}$.  We consider the generating function
\begin{eqnarray*}
X(G,q) & := & \sum_{\overline{S} \in 2^{[n]}/G}q^{|S|} \\ & = & \sum_{i=0}^{n}\left| {{[n]} \choose i}/G \right|q^i,
\end{eqnarray*}
counting $G$-orbits on $2^{[n]}$ according to the cardinalities of their representatives.  (Here ${{n} \choose {i}}/G$ is the set of $G$-orbits on ${{n} \choose {i}}$.)  The polynomial $X(G,q)$ has been well-studied via algebraic means, perhaps first in the work of Redfield and P\'olya.  Key in this study is the introduction of a vector space $V$ over an arbitrary field $\FF$ with basis $\{e_0,e_1\}$.  Having fixed $G \leq \Sym_n$, we define the tensor product $\CC(\FF)=V^{\otimes n}$.  For $S \subseteq [n]$, we define
\[
e_S:=e_{i_1} \otimes \cdots \otimes e_{i_n} \in \CC(\FF),
\]
where
\[
i_j:=\left\{ \begin{array}{ll} 1 & \mbox{if }j \in S, \\ 0 & \mbox{if } j \not\in S. \end{array} \right.
\]
Then $\{e_S:S \subseteq [n]\}$ is a basis for $\CC(\FF)$ that is permuted by $G$ according the the action of $G$ on $2^{[n]}$ and thus we obtain a linear action of $G$ on $\CC(\FF)$.  The vector space $\CC(\FF)$ admits a natural grading
\[
\CC(\FF)=\bigoplus_{i=0}^{n}\CC(\FF)_i,
\]
where
\[
\CC(\FF)_i:=\FF\{e_S:S \in {{[n]} \choose {i}}\}.
\]
Now $G$ acts linearly on Each $\CC(\FF)_i$, and we can define the spaces $G$-{\it invariants}
\[
\CC(\FF)_i^G:=\{v \in \CC(\FF)_i:g(v)=v \mbox{ for all } g \in G\}
\]
and the space of $G$-{\it coninvariants}
\[
(\CC(\FF)_i)_G:=\CC(\FF)_i/\FF\{u-g(u):u \in \CC(\FF)_i, g \in G\}.
\]
We define also
\[
\CC(\FF)^G:=\bigoplus_{i=0}^{n}\CC(\FF)_i^G
\]
and
\[
\CC(\FF)_G:=\bigoplus_{i=0}^{n}(\CC(\FF)_i)_G.
\]
Note that, for $0 \leq i \leq n$,
\begin{equation} \label{orbinv}
\dim \CC(\FF)_i^G=\left| {{[n]} \choose {i}}/G\right| = \dim (\CC(\FF)_i)_G.
\end{equation}
Indeed, for each $\Omega \in {{[n]} \choose {i}}/G$, set
\[
e_\Omega:=\sum_{S \in \Omega}e_S.
\]
Then $\{e_\Omega:\Omega \in {{n} \choose {i}}/G\}$ is a basis for $\CC(\FF)_i^G$.  Similarly, if we define $\overline{e_S}$ to be the coset represented by $e_S$ in $(\CC(\FF)_i)_G$, then we get a basis $\{\overline{e_S}\}$ for $(\CC(\FF)_i)_G$ by selecting one representative $S$ from each $G$-orbit on ${{[n]} \choose {i}}$.

See \cite[Corollary 6.2]{RSW} for the next result, which is due to N. G. de Bruijn.

\begin{theorem}
(de Bruijn)
\label{deBruijn-theorem}
$X(G, -1)$ is the number of $G$-orbits $\bar{S}$ of subsets of $[n]$ 
which are self-complementary in the sense that $\bar{S} = \overline{[n] \backslash S}$.
\end{theorem} 

For example, if $G=\langle (1234) \rangle $, then $\overline{S} = \{ \{ 1,3\} ,\{ 2,4\} \}$ is a 
self-complementary orbit. 

We include a proof of Theorem \ref{deBruijn-theorem}, both because it 
is quite elegant and compact, and also so that it may be compared with
our homological approach.

\begin{proof}
The action of $GL_2(\QQ)$ on $V$ induces an action on $\CC(\QQ)$, defined by
\[
h(v_1 \otimes \cdots \otimes v_n)=h(v_1) \otimes \cdots \otimes h(v_n).
\]
This action commutes with the action of $G$, and it follows that $\CC(\QQ)^G$ is stabilized by $GL_2(\QQ)$.

Consider the two elements of $GL_2(\QQ)$,
$$
g=\left[ \begin{matrix} 0 & 1\\
                  1 & 0\\
\end{matrix} \right], \qquad
g'=\left[ \begin{matrix} 1 & 0\\
                  0 & -1\\
\end{matrix} \right].
$$
Note that $g, g'$ are conjugate in $GL_2(\QQ)$. So, they will 
act by conjugates on $\CC(\QQ)^G$, and hence have the same trace on $\CC(\QQ)^G$.

Since $g$ swaps $e_0$ and $e_1$ in each copy of $V$, it swaps
the basis element $e_{\bar{S}}$ in $\CC(\QQ)^G$
with the basis element $e_{\overline{[n]\backslash S}}$; thus 
it acts with trace equalling the number of self-complementary 
$G$-orbits of subsets.

On the other hand, $g'$ is the specialization
to $q=-1$ of the diagonal matrix
$$
g(q)=\left[ \begin{matrix}
                  1 & 0\\
                  0 & q\\
\end{matrix} \right].
$$
Since $g(q)$ acts on $V$ by  sending $e_0$ to itself and $e_1$ to $qe_1$, the induced action
of $g(q)$ on $\CC(\QQ)^G$ has $e_{\bar{S}}$ as an eigenvector 
with eigenvalue $q^{|S|}$.  Hence $g(q)$ acts on
$\CC(\QQ)^G$ with trace $X(G,q)$.  In particular, $g'=g(-1)$ acts with trace $X(G,-1)$.
\end{proof}

Now define the {\it up} and {\it down} maps $U, D: \CC(\FF) \rightarrow \CC(\FF)$
by 
$$
\begin{aligned}
U(e_S)& = \sum_{i \in [n] \backslash S} e_{S \cup \{i\}}\\
D(e_S)& = \sum_{j \in S} e_{S \backslash \{j\}}.
\end{aligned}
$$
Note that $U, D$ both commute with the $G$-action.  As a consequence, they give well-defined
maps on the graded vector spaces $\CC(\FF)^G$ and $\CC(\FF)_G$.  Indeed, $D$ and $U$ map $\CC(\FF)^G_i$ to $\CC(\FF)^G_{i-1}$ and $\CC(\FF)^G_{i+1}$, respectively, and a similar statement holds for coinvariants.  In addition to Theorem \ref{deBruijn-theorem}, the algebraic approach yields the following result.

\begin{theorem} (\cite{Sta0,Proctor})
\label{sl2-theorem}
The polynomial $X(G,q)$ has symmetric, unimodal coefficients.
\end{theorem}

Indeed, one can verify that the Lie subalgebra of ${\mathfrak g}{\mathfrak l}(\CC(\complexes))$ generated by $U$ and $D$ is isomorphic with ${\mathfrak s}{\mathfrak l}_2(\complexes)$ and obtain Theorem \ref{sl2-theorem} from classical results in representation theory.  

We use $D$ and $U$ for a different purpose, namely, to define the desired chain complexes, which have Euler characteristic $X(G,-1)$.  As mentioned in the introduction, this will allow us to prove that the $q=-1$ phenomenon occurs in the case of partitions in a rectangle without the use of Theorem \ref{deBruijn-theorem}.

\begin{proposition}
\label{four-complexes}
\begin{enumerate} 
\item The maps induced by $D$ on the graded vector spaces $\CC(\FF_2)^G$ and $\CC(\FF_2)_G$ are both chain maps, i.e.,  $D^2=0$ in each
case.  Likewise, the maps induced by $U$ on $\CC(\FF_2)^G$ and $\CC(\FF_2)_G$  are both cochain maps, i.e., $U^2 = 0$. 

\item Each of the  four complexes $(\CC(\FF_2)^G,D)$, $(\CC(\FF_2)_G,D)$, $(\CC(\FF_2)^G,U)$, $(\CC(\FF_2)_G,U)$ has Euler characteristic $X(G,-1)$.
\end{enumerate}
\end{proposition}

\begin{proof}
The maps $D, U$ on $\CC(\FF_2)$ coincide
with the boundary and coboundary maps $\partial_i$ and
$\partial^i$ in the usual simplicial chain complex for a simplex
having vertex set $[n]$. Hence $D^2=U^2=0$, and the same holds
after taking $G$-invariants or $G$-coinvariants.  This gives (1).  Assertion (2) follows from (\ref{orbinv}).
\end{proof}

As above, given a $G$-orbit $\bar{S}$, we set $e_{\bar S}:=\sum_{S' \in \bar{S}}e_{S'}$ and write $\overline{e_S}$ for
the basis element of $\CC(\FF)_G$ corresponding to $\bar{S}$.  For a finite cochain complex $\CC$,  $\CC^\op$ will denote
the opposite chain complex that one obtains by reindexing in the opposite order  and reversing all the arrows.

\begin{proposition} \label{same}
The map sending $e_{\bar S} \mapsto \overline{e_S}$ induces isomorphisms
of complexes
$$
\begin{aligned}
(\CC(\FF_2)^G,U) &\cong \Hom_{\FF_2}((\CC(\FF_2)_G,D), \FF_2)\\
(\CC(\FF_2)^G,D) &\cong \Hom_{\FF_2}((\CC(\FF_2)_G,U), \FF_2).
\end{aligned}
$$
The map sending $e_{\bar S} \mapsto \overline{e_{[n]\backslash S}}$
induces isomorphisms of complexes
$$
\begin{aligned}
(\CC(\FF_2)^G,D) &\cong (\CC(\FF_2)^G,U)^{\op}\\
(\CC(\FF_2)_G,D) &\cong (\CC(\FF_2)_G,U)^{\op}.
\end{aligned}
$$
\end{proposition}

\begin{proof}
Let $\bar{S}, \bar{T}$ be $G$-orbits of subsets with $|T|=|S|+1$.
Then the boundary map coefficient $D_{\bar{S},\bar{T}}$ in $\CC(\FF_2)^G$ is the number of elements
$T'$ in the orbit $\bar{T}$ which {\it contain} the fixed set $S$ in $\bar{S}$.
Meanwhile the boundary coefficient $D_{\bar{S},\bar{T}}$ in $\CC(\FF_2)_G$ is the number of elements
$S'$ in the orbit $\bar{S}$ which are {\it contained in}
the fixed set $T$ in $\bar{T}$. There are similar formulae for the
coefficients $U_{\bar{S},\bar{T}}$ in the two complexes.  The isomorphisms are not
hard to verify, using the fact 
that set-complementation is an inclusion-reversing
bijection.
\end{proof}

In light of the previous proposition,
one may consider any one of the four
complexes, as its homology determines the homology of the others
(either by turning it around in homological degree, or by taking dual $\FF_2$-vector
spaces, or both).

%

\begin{theorem}
\label{odd-acyclic}
The complex $(\CC(\FF_2)^G , D)$ is acyclic 
when $n$ is odd.
More generally, it is  
acyclic whenever
$G$ has at least one orbit of odd cardinality in its action on $[n]$.
If $(\CC(\FF_2)^G, D)$ is not acyclic, then 
$H_0(\CC(\FF_2)^G, D)=\FF_2$ and $H_1(\CC(\FF_2)^G, D)=0$.
\end{theorem}

\begin{proof}
Let $S$ be the set of elements of $[n]$ belonging to a $G$-orbit of odd size, which in 
particular ensures that $S$ is $G$-stable
(although not necessarily pointwise fixed
by $G$).  Then one forms {\it $S$-masked} version of the up map in $\CC(\FF_2)$
as follows:
$$
\begin{aligned}
U^{(S)}(e_T) &:=\sum_{i \in S \backslash T} e_{T \cup \{i\}} 
\end{aligned}
$$

The $G$-stability of $S$ implies that $U^{(S)}$ commutes with the $G$-action.  Indeed, the crucial point in all the requisite calculations is that
$g(S)=S$, so that, for example,
$$
\begin{aligned}
S \backslash g(T) &= g(S \backslash T),\\
S \cap g(T) &= g(S \cap T).
\end{aligned}
$$
Hence one gets an induced $S$-masked up map $U^{(S)}$
on $\CC(\FF_2)^G$.

A straightforward calculation shows that
$$
\begin{aligned}
(D \cdot U^{(S)}- U^{(S)} \cdot D)(e_T)
&= (|S\backslash T|-|S \cap T|) \cdot e_T \\
&= (|S|-2|S \cap T|) \cdot e_T.
\end{aligned}
$$
As we are working in characteristic two, this gives
$$
D \cdot  U^{(S)} + U^{(S)}\cdot D = |S|\cdot I.
$$

Thus when $|S|$ is odd, the $S$-masked up map $U^{(S)}$ gives
an algebraic chain-contraction, showing that $(\CC(\FF_2)^G,D)$ is acyclic.  
%

Next we verify the assertions about $H_0$ and $H_1$ by closely examining 
the first 
few boundary maps in $(\CC(\FF_2)^G,D)$.

The boundary map $D$ out of $\CC(\FF_2)_0^G=\FF_2$ is {\it always} the
zero map, regardless of $G$.  If all $G$-orbits have even cardinality,
the boundary map $D$ out of  $\CC(\FF_2)_1^G$ will also be the zero map, so the assertion
about $H_0$ follows.

It remains to show that when all $G$-orbits have even cardinality, the map $D$ out of
$\CC(\FF_2)_2^G$ is surjective.  But for this we can work within each $G$-orbit $X$ on $[n]$.
That is, it suffices to show that there is some $G$-orbit $Y=\overline{\{i,j\}}$ 
of pairs with
$i,j \in X$ for which $D(e_Y)$ has coefficient $1$ on $e_X$, not zero.  However, fixing $X$,
one can see that the sum of all of such boundary map coefficients
incident to $X$ and {\it coming from $G$-orbits of pairs contained in $X$}
will be $|X|-1$, an odd number.  Thus
one of them must be non-zero in $\FF_2$, as desired.
\end{proof}


It follows from Theorem \ref{deBruijn-theorem} and Proposition~\ref{four-complexes}(2) that the Euler characteristic of all of the complexes under consideration is non-negative.  Given this fact and the calculations of $H_0, H_1$ in Theorem~\ref{odd-acyclic},
one might be tempted to make the conjecture (true up through $n=5$)
that the homology of $\CC(\FF_2)^G$ is always concentrated in even dimensions.
However, we shall see in Section ~\ref{not-even}  that this is not the case.  In Section \ref{rect}
we will nonetheless confirm this behavior for wreath products of symmetric groups in their natural imprimitive actions, and in Section 
~\ref{not-even} we provide evidence that this holds for cyclic groups as well.

\section{A Morse matching lemma} \label{morse-sec}


In this section we prove Lemma \ref{Morse-lemma}, which is an algebraic version of the main result in R. Forman's discrete Morse theory.   Similar results appear in the work J\"ollenbeck-Welker and that of Sk\"oldberg (see \cite{JW} and \cite{Sk}), in which much more extensive versions of algebraic discrete Morse theory are developed.  Lemma \ref{Morse-lemma} will be applied in Sections \ref{rect}  and \ref{plane} to obtain homology bases for chain complexes under consideration, and thereby compute their Euler characteristics.  

\begin{definition}
Let $P$ be a graded poset and, for $i \geq 0$, let $P_i$ be the set of elements of rank $i$ in $P$.  Let $(\CC,d)$ be a chain complex of $\FF$-vector spaces $\CC_i$ such that $\CC_i \neq 0$ if and only if $P_i \neq \emptyset$.  We say that $P$ {\it supports} $(\CC,d)$ if each $\CC_i$ has basis (indexed by) $P_i$ and, for $p \in P_i$ and $q \in P_{i-1}$, the coefficient
$d_{p,q}$  of $q$ in $d_i(p)$ is zero unless $q<_P p$.

We consider the Hasse diagram of $P$ as a graph.  A (partial) matching $M$ on this Hasse diagram $H$ is said to be 
an {\it acyclic matching}, or a {\it Morse matching}, if the digraph $D(P,M)$, obtained by first directing all edges in $H$ downward and then reversing all the directions of the edges in $M$, is acyclic (see e.g. \cite{Ch}, \cite{BBLSW}, \cite{He}).

Given an acyclic matching $M$ on $P$, let the subsets
$$
P^{\matched}, P^{\unmatched}, P^{\matched}_i, P^{\unmatched}_i
$$
respectively denote the $M$-matched and $M$-unmatched elements in $P$, and the same sets
restricted to rank $i$. The elements of $P^{\unmatched}$ are called {\it critical elements}.
We write $q=M(p)$ if the edge $\{p,q\}$ of $H$ lies in $M$.
We write $<_{D(P,M)}$ for the partial order on $P$ that results from taking the transitive closure of $D(P,M)$.
\end{definition}

\begin{lemma}
\label{Morse-lemma}
Let $P$ be a graded poset supporting an algebraic complex $(\CC,d)$ and assume $P$ has
a Morse matching $M$ such that for  all $q=M(p)$ with $q < p$, one has $d_{p,q} \in \FF^\times$.
Let $Q_i$ be the set of all $p \in P_i^M$ such that $M(p) \in P_{i-1}$.  Then
\begin{enumerate}
\item[(i)] $\dim H_i(\CC,d) \leq |P^{\unmatched }_i|$.
\item[(ii)] If $|Q_i|  = rank (d_i)$ for every $i$, then
$\dim H_i = |P^{\unmatched }_i| $.  (For example, this condition is met
whenever all unmatched poset elements live in 
ranks of the same parity.)
\item[(iii)] If $d_{q,p} = d_{p,r} = 0$ for all $p\in P^{\unmatched } $ and all $q,r\in P$, 
then the homology $H(\CC,d)$ has $\FF$-basis $\{e_p: p \in P^{\unmatched}\}$.
\end{enumerate}
\end{lemma}

\begin{proof}
To prove (i), we will show that the matrices of the boundary maps $d_i$ have sufficiently large
ranks  by showing that they have large, nonsingular square submatrices.
Order $Q_i$ (as defined in the statement of this lemma) 
as $q_1,\ldots,q_r$ by any linear extension of the partial order $<_{D(P,M)}$.  We claim that the
square submatrix of $d_i$ having columns indexed $q_1,,\ldots,q_r$ and
rows indexed by  $M(q_1),\ldots,M(q_r)$ is invertible upper-triangular.

To prove the claim, note that the hypothesis $d_{p,q}\in \FF^{\times } $ for $q < p$ 
implies that the diagonal entries
of this square matrix are all in $\FF^\times$, since $q=M(p)$ implies $q<p$ or $p<q$.
Hence one only needs to verify
upper-triangularity.  So assume that the boundary map $d_i$ has $(d_i)_{q_j,M(q_k)} \neq 0$
for some $k \neq j$.
Since the complex $\CC$ is supported on $P$, we see that $q_j >_P M(q_k)$ and hence
$D(P,M)$ has a directed edge $q_j \rightarrow M(q_k)$.  There is also the matching edge in $D(P,M)$,
directed upward as $M(q_k) \rightarrow q_k$, and thus by transitivity, $q_j <_{D(P,M)} q_k$.  Hence $j < k$, yielding the claim.

The claim implies that $\rank (d_i) \geq |Q_i|$ for all $i$.  As usual letting 
$Z_i = \ker d_i$ and $B_i = \im d_{i+1}$, note that
$$
\begin{aligned}
\dim H_i & = \dim Z_i - \dim B_i \\
         & = |P_i| - (\rank (d_i) + \rank (d_{i+1}) ) \\
         &\leq |P_i| - (|Q_i| + |Q_{i+1}|) \\
         & =    |P_i| - |P^{\matched}_i|\\
         & = |P^{\unmatched}_i|
\end{aligned}
$$
as desired.

Now the hypothesis in (ii) makes the weak inequality into an equality in the above string of 
equalities and weak inequalities, implying the desired equality in (ii).  To prove (iii), note that
the hypothesis $d_{p,r} = 0$ for all $r<p$
ensures that $ \{ e_p|p\in P_i^{\unmatched } \} $ is contained in the kernel of
$d_i$ for each $i $, while $d_{q,p} = 0$ for all $q>p$ implies linear independence of
this set in $H_i(\CC,d)$.  Since $|P_i^{\unmatched }| \ge \dim H_i (\CC, d)$, this set also must span $H_i(\CC, d)$, hence
is a homology basis.
\end{proof}







\section{Application to partitions in a rectangle}\label{rect}


Let $P(k,\ell)$ be the set of integer partitions  (see for example \cite{And})  
whose Ferrers diagrams lie
inside a $k\times \ell$ rectangle. Any $\lambda\in P(k,\ell)$ may be written as
$\lambda=(\lambda_1,\ldots,\lambda_k)$ with
$$
\ell \geq \lambda_1 \geq \cdots \geq \lambda_k \geq 0.
$$
$P(k,\ell)$ is a ranked poset under inclusion of Ferrers diagrams. The 
rank generating function of $P(k,\ell)$ is given by the $q$-binomial coefficient
$$
X(k,\ell,q)=
\left[ \begin{matrix} k+\ell\\k
\end{matrix}\right]_q.
$$
It is well-known that the number of self-complementary partitions 
in $P(k,\ell)$ is $X(k,\ell,-1).$  Here we will prove this fact by applying Proposition \ref{four-complexes}(2) and Lemma \ref{Morse-lemma}, thus avoiding the use of Theorem \ref{deBruijn-theorem}.


First let us check that Proposition \ref{four-complexes} may be applied.  Consider the partition of the set $C$ of all $k\ell$ boxes in the $k \times \ell$ rectangle into $k$ subsets of size $\ell$, each subset consisting of the boxes in a given row of the rectangle. The wreath product $G=\Sym_\ell \wr \Sym_k$ is the full stabilizer of this partition in $\Sym_C$.  An element of $G$ first permutes rows wholesale and then permutes boxes within each row independently.  It is not hard to see that each $G$-orbit on the set of subsets of $C$ contains (the Ferrers diagram of) a unique partition and that $(\CC(\FF_2)^G,D)$ is supported on $P(k,\ell)$.

We will use Lemma \ref{Morse-lemma} to show that the homology of $(\CC(\FF_2),D)$ is concentrated in even dimensions and exhibit a basis for the homology.  Then we will describe a bijection between this basis and the set of self-complementary partitions in $P(k,\ell)$.

One may check directly that the entries in the boundary
maps $D$
take the following form. 
\begin{proposition}
\label{wreathdown}
If $\mu$ is obtained from $\lambda$
by reducing the part $\lambda_i$ to $\lambda_i-1$, then
\begin{equation}
\label{wreath-boundary-coefficients}
D_{\lambda,\mu} = (\ell - \lambda_i + 1) (\mult_\lambda(\lambda_{i}-1)+1)
\end{equation}
where $\mult_\lambda(s)$ is the multiplicity of the part $s$ in $\lambda$.
\end{proposition}

In order to define our Morse matching $M$ on $P(k,\ell)$, we need to specify two parts of each partition therein. Given $\lambda \in P(k,\ell)$,
\begin{itemize}
\item[(1)] $A(\lambda)$ is the last nonzero part of $\lambda$ having the same parity
as $\ell$, if such a part exists.  If no such part exists, then
$A(\lambda)=\infty$.  Also,

\item[(2)] $B(\lambda)$ is the first occurrence of the smallest part of $\lambda$
having odd multiplicity and opposite parity to $\ell$, if such a part
exists.    If no such part exists, $B(\lambda)=\infty$.
\end{itemize}

Note that while $A(\lambda)>0$ always holds, it is possible that $B(\lambda)=0$.

Now we define $M(\lambda)$ for each $\lambda \in P(k,\ell)$ satisfying 
$$
min\{A(\lambda),B(\lambda)\}<\infty.
$$
We will see that $M$ determines a partial matching on $P(k,\ell)$.

\begin{definition} Suppose that $\lambda\in P(k,\ell)$ and $min\{A(\lambda),B(\lambda)\}<\infty.$
The matching $M(\lambda)=\mu$ is defined as follows.
\begin{enumerate}
\item If $A(\lambda)< B(\lambda)$, then $ \mu$ is obtained from $\lambda$ 
by replacing $A(\lambda)$ with $A(\lambda)-1$. 
\item If $A(\lambda)>B(\lambda)$, then $ \mu$ is obtained from $\lambda$ 
by replacing $B(\lambda)$ with $B(\lambda)+1$.
\end{enumerate}
\end{definition}

Note that $A(\lambda)$ and $B(\lambda)$ have opposite parities if they are both finite and are therefore distinct.
Since $A(\lambda)>0$ and $B(\lambda)$ has the opposite parity to the largest allowed part $\ell$, we see that $M(\lambda) \in P(k,\ell)$ whenever $\min\{A(\lambda),B(\lambda)\}<\infty$. So the function $M$ is a well-defined on its domain, and $M(\lambda)$ either covers or is covered by $\lambda$ in $P(k,\ell)$.

Next we check that $M$ is an involution on $P(k,\ell).$
If $A(\lambda)<B(\lambda),$ then 
$A(M(\lambda))\ge A(\lambda)$ and $B(M(\lambda))=A(\lambda)-1,$  because the parts of size  
$A(\lambda)-1$ have odd multiplicity in $M(\lambda)$. Thus applying $M$ again we see that
$M(M(\lambda))=\lambda.$ If $A(\lambda)>B(\lambda),$ then 
$B(M(\lambda))\ge B(\lambda)+2$, and 
$A(M(\lambda))=B(\lambda)+1$, so $M(M(\lambda))=\lambda.$


The $M$-unmatched partitions are those such that $A(\lambda)=B(\lambda)=\infty.$

\begin{proposition}
\label{unmatched}
The $M$-unmatched partitions $\lambda \in P(k,\ell)$ are exactly those satisfying
\begin{enumerate}
\item every nonzero part of $\lambda$ has even multiplicity and opposite parity to $\ell$, and \item if $\ell$ is odd then $0$ has even multiplicity in $\lambda$. \end{enumerate}
\end{proposition}

Note that if $k$ and $\ell$ are both odd, $P^{\unmatched}=\varnothing$.  Indeed, if $\ell$ is odd and
a critical cell  $\mu$ of $M$ exists, then $\mu$ has $k$ parts (including 0), 
all of which are even, with even multiplicity. So $k$ must be even.

\begin{theorem}
\label{rectangle-thm}
The matching $M$ on the poset $P=P(k,\ell)$ is acyclic, with
the partitions $\lambda $ in $P^{\unmatched}$  being those 
described in Proposition~\ref{unmatched}.
Moreover, the homology of $\CC((\mathbb{F}_2)^G,D)$ 
is concentrated in even dimensions. 
\end{theorem}

\begin{proof}
We need to verify the non-vanishing condition $D_{p,q}\in {\mathbb{F}}_2^\times$,
the hypotheses in (ii) and (iii) of Lemma~\ref{Morse-lemma}, 
and the acyclicity of $M$.

Let $p$ be $M$-matched with $M(p) = q < p$. 
We must show that $D_{p,q}=1\neq 0\in {\mathbb{F}}_2.$
The part $A(p)$ which is reduced by one by the matching $M$ has the same parity as $\ell$, so 
$\ell-A(p)+1=1\in {\mathbb{F}}_2.$ Since all parts of $p$ smaller than $A(p)$ have 
parity distinct from that of  $\ell$, and $B(p)>A(p),$ the multiplicity of $A(p)-1$ is even, so 
$mult_{p}(A(p)-1)+1=1\in {\mathbb{F}}_2.$ Thus by Proposition~\ref{wreathdown},
$D_{p,q}=1\neq 0\in {\mathbb{F}}_2.$

From Proposition~\ref{unmatched}, each unmatched $\lambda$ is a partition 
of an even integer, so (ii) is clear.

Let $p$ be $M$-unmatched. By Proposition~\ref{unmatched} all parts of $p$ 
have opposite parity from $\ell$
so Proposition~\ref{wreathdown} implies $ D_{p,r}=0\in {\mathbb{F}}_2$. 
Similarly $D_{q,p}=0\in {\mathbb{F}}_2,$ because, with $q_i$ as defined in Proposition \ref{wreathdown}, the multiplicity of 
$q_i-1$ in $q$ must be odd.

Let us now check acyclicity of $M$.  Suppose there is a
directed cycle $C$. Then $C$ alternates
$M$-edges traversed upward and non-$M$ Hasse 
edges traversed downward.  Let $E$ be the smallest part incremented or decremented 
on $C$. $E$ cannot be decremented, since all increments would then be 
of larger integers, and $C$ is a cycle. So $E$ is incremented, has the same parity as $\ell$, and 
all parts less than $E$ are fixed in $C$. When $E=B(\lambda)$ was incremented, $\lambda$ 
has no other parts  less than $E$ of the same parity  as $\ell$. So 
there are no parts less than $E$ of the same parity as $\ell$ anywhere on $C$.  Moreover 
all parts less than $E$ of the opposite parity from $\ell$ always have even multiplicity. 
The multiplicity of $E$ in $\lambda$ is odd. Find the first edge $f$ in $C$ 
where $E+1$ is decremented to $E$ in $\mu$ along a non-$M$ Hasse edge. At this time 
$A(\mu)=E+1$ because $\mu$ has no smaller parts with the same parity as $\ell.$  Also
$B(\mu)\ge E+2,$ because the multiplicity of $E$ in $\mu$ is even. Hence the edge
$f$ is an $M$-edge, which is a contradiction.
\end{proof}


\begin{theorem}
There is a bijection $\phi$ between the set of critical cells of $M$ and the set of self-complementary
partitions in a $k\times \ell $ rectangle.
\end{theorem}

\begin{proof}
Let $\lambda \in P(k,\ell)$.  Walking from the southwest corner of the $k \times \ell$ rectangle to the northeast corner along the southeast boundary of the Ferrers diagram of $\lambda$ (which includes the parts of the western and northern borders of the rectangle that are not contained in the Ferrers diagram), we take $k$ steps of length one to the north and $\ell$ steps of length one to the east, in some order.  Recording these steps in the order they are taken, we get a word $W(\lambda)$ of length $k+\ell$ in the alphabet $\{N,E\}$.  Collecting like terms, we write
\begin{equation} \label{word}
W(\lambda)=N^{k_1}E^{\ell_1}\ldots N^{k_t}E^{\ell_t},
\end{equation}
where $k_1,\ell_t \geq 0$ and all other $k_i,\ell_j$ are positive, $\sum_{i=1}^{t}k_i=k$ and $\sum_{j=1}^{t}\ell_j=\ell$.  

Note that if $\mu$ is complementary to $\lambda$ in our rectangle then $W(\mu)$ is the word $W(\lambda)^{\rev}$ obtained by writing $W(\lambda)$ backwards.  Therefore, $\lambda$ is self-complementary if and only if $W(\lambda)=W(\lambda)^{\rev}$.

We define $\phi$ depending on the parities of $k,\ell$ as follows.  If $k,\ell$ are both odd then both sets are empty and there is nothing to do.  

Now fix $k,\ell$ such that $k\ell$ is even, and let $\lambda$ be a critical cell of $M$.

Say $k$ is even and $\ell$ is odd.   It follows from Proposition \ref{unmatched} that, with $W(\lambda)$ written as in (\ref{word}),  each $k_i$ is even, $\ell_j$ is even for $1 \leq j \leq t-1$ and $\ell_t$ is odd.  We set
\[
W^\prime(\lambda):=N^{k_1/2}E^{\ell_1/2}\ldots N^{k_t/2}E^{(\ell_t-1)/2},
\]
and define $\phi(\lambda)$ to be the partition satisfying
\[
W(\phi(\lambda))=W^\prime(\lambda)E(W^\prime(\lambda))^{\rev}.
\]
Now say $k$ and $l$ are both even.  Then, in $W(\lambda)$, each $k_i$ is even.  If $\lambda$ is the empty partition then $W(\lambda)=N^kE^\ell$ and we define $\phi(\lambda)$ to be the partition satisfying
\[
W(\phi(\lambda))=N^{k/2}E^\ell N^{k/2}.
\]
Otherwise, in $W(\lambda)$, $\ell_1$ and $\ell_t$ are odd and all other $\ell_j$ are even.  We set
\[
W^\prime(\lambda):=N^{k_1/2}E^{(\ell_1+1)/2}N^{k_2/2}E^{\ell_2/2}\ldots N^{k_t/2}E^{(\ell_t-1)/2},
\]
and define $\phi(\lambda)$ to be the partition satisfying
\[
W(\phi(\lambda))=W^\prime(\lambda)W^\prime(\lambda)^{\rev}.
\]
Say $k$ is odd and $\ell$ is even.  Then $k_1$ is odd and $k_i$ is even for $2 \leq i \leq t$.  If $\lambda$ is the empty partition, so $W(\lambda)=N^kE^\ell$, we define  $\phi(\lambda)$ to be the partition satisfying
\[
W(\phi(\lambda))=N^{(k-1)/2}E^{\ell/2}NE^{\ell/2}N^{(k-1)/2}.
\]
Otherwise, $\ell_1$ and $\ell_t$ are odd, while $\ell_j$ is even for $2 \leq j \leq t-1$.  We set
\[
W^\prime(\lambda):=N^{(k_1-1)/2}E^{(\ell_1+1)/2}\ldots N^{k_t/2}E^{(\ell_t-1)/2},
\]
and define $\phi(\lambda)$ by
\[
W(\phi(\lambda))=W^\prime(\lambda)N(W^\prime(\lambda))^{\rev}.
\]
We leave it to the reader to find $\phi^{-1}$ in each case, thus confirming that $\phi$ is a bijection.

\end{proof}

\section{Application to plane partitions of bounded value in a rectangle}\label{plane}

Plane partitions  of bounded value inside a rectangle 
(see e.g. \cite{And}), or equivalently 3-dimensional partitions 
which lie inside an $r\times c\times t$ box, generalize
the objects  considered in the previous section, namely the
integer partitions which lie inside an $r\times c $ rectangle.  Specifically,
these integer partitions comprise  the $t=1$ case.
The generating function for all plane partitions inside an $r\times
c\times t $ box, counted by volume, is given by the Macmahon box formula \cite{And}
$$
X(r,c,t,q)=\prod_{i=1}^r\prod_{j=1}^c \frac{1-q^{r+t+j-i}}{1-q^{r+c-j-i+1}}.
$$
Stembridge \cite{Ste3} proved that the number of self-complementary plane partitions inside 
the $r\times c\times t $ box is $X(r,c,t,-1).$

In this section, we construct an ${\mathbb{F}}_2$-complex $(\CC(r,c,t),d)$ whose
$i$-dimensional cells are indexed by plane partitions of volume $i$ in a $r\times
c\times t $ box.  We prove homology concentration in Theorem \ref{box-thm}
and also give a homology basis comprised of objects that are known to be 
equinumerous with the set of self-complementary plane partitions.  


Rather than using the set of plane partitions as our indexing set, it will be more 
convenient instead to index bases for our chain spaces by another set of equal
cardinality, namely the set of semistandard Young tableaux (SSYT) of shape
$\lambda = (c)^r$ in which all entries lie weakly between $0$ and $r+t-1$.  We 
think of a plane partition inside an $r\times c$ rectangle as a Ferrers diagram of
shape $\lambda $ with weakly increasing rows and columns with integers 
entries lying weakly between $0$ and $t$.  A bijective map sending such a 
plane partition to a SSYT is obtained by adding $i-1$ to each entry in the $i$-th
row for each $i\in [r]$.


%

%

We now define the complex $(\CC(r,c,t),d)$, with coefficients in ${\mathbb{F}}_2$, 
by letting the chain group generators be indexed by 
the SSYT of $r\times c$ rectangular shape with entries
between $0$ and $t+r-1,$ denoted $SSYT(r,c,t).$  We partially order this set
$SSYT(r,c,t)$ by setting 
$T_1\le T_2$ if each entry of 
$T_1$ is no larger than the corresponding entry of $T_2.$  The poset 
$SSYT(r,c,t)$ is ranked with rank function 
\begin{equation}
\label{rank-eqn}
rank(T)= \left({\text{sum of entries of }} T\right) -c\binom{r}{2}=
||T||-c\binom{r}{2}.
\end{equation}

Define the boundary map $d$ as follows.  For $T\in SSYT(r,c,t),$
$dT$ is a sum over all possible $S$ in $SSYT(r,c,t)$ obtained by subtracting one 
from the leftmost copy in some row $R$ of some odd value $2k+1$ such that $2k+1$
has the following property: there are an odd number of copies of 
$2k+1$ in row $R$ that do not have $2k$ immediately above them in row
$R-1$.  In other words, $$dT = \sum_{T' \in B(T)} T'$$
for $B(T)$ the set of $SSYT(r,c,t)$ obtained by subtracting one from an odd
entry $\lambda_{i,j}$ of $T$ located at position $(i,j)$, with the 
further requirement that there are an odd number of occurrences of 
the value $\lambda_{i,j}$ in row $i$ that do not have the value
$\lambda_{i,j} - 1$ in the position just above them in row $i-1$.

\begin{proposition} 
$(\CC(r,c,t),d)$ is a chain complex, i.e., $d^2 = 0$.
\end{proposition}

\begin{proof}
Any pair of distinct values that may both be decremented when we
apply $d$ twice to $T$ 
may be decremented in either order, so that the resulting SSYT with
the two values decremented occurs
with coefficient of 0 in $d^2 T$, because we are working  mod 2. 
One fact this relies upon is that
decrementing an odd value will not impact the parity calculation involved in the 
description of $d$ for other odd
values in the same row or for odd values in other rows.

The parity requirements preclude $d^2$  from decrementing a
single value by two or decrementing the same value at two different
positions in the same row.   
\end{proof}
%
%

\begin{definition}\label{3d-matching-definition}
\label{M-def}
Let $T$ be in $SSYT(r,c,t).$ 
Consider the rows in $T$ having at least one of the 
following properties:
\begin{enumerate}
\item
There is an odd value $i$ such that there are an odd number of  
occurrences of $i$ in the row which do not have $i-1$ immediately above
them in the prior row,
\item
There is an even value $i<r+t-1$ not having the value $i+1$ immediately below it such that the
value $i+1$  occurs an even (possibly zero) number of times in $i$'s
row in positions that do not have $i$ immediately above them.
\end{enumerate}
In (1) and (2), any condition referring to row $0$ or row $r+1$ is vacuously fulfilled. 

If there is some row of $T$ satisfying (1) or (2), define $M(T) \in SSYT(r,c,t)$ to be the 
tableau obtained from $T$ by

\begin{enumerate}
\item
choosing the smallest row $R$ of $T$ containing some $i$ that satisfies (1) or (2)
\item 
choosing the smallest $i$ in row $R$ such that $i$ satisfies (1) or (2)
\item
subtracting one from the leftmost copy of $i$ in row $R$ if $i$ is odd, and adding 
one to the rightmost copy of $i$ in $R$ if $i$ is even.
\end{enumerate}

\end{definition}

Next we verify that $M$ is well-defined. 
In case (1) decrementing the odd entry $i$ is allowed, and 
$M(T)\in SSYT(r,c,t).$ In case (2), incrementing the even entry $i$ is allowed,
and $M(T)\in SSYT(r,c,t)$ because $i+1\le r+t-1.$ 

Now let us show that $M$ is a matching, namely that  $M(M(T))=T$ whenever some row of
$T$ satisfies one of the conditions (1), (2) of Definition ~\ref{3d-matching-definition}. 
Suppose that $M(T)$ is obtained from $T$ by decrementing an odd entry $i$ in row $R$.
Then in  $M(T)$, row $R$ satisfies the condition (2) with respect to $i-1$.  Indeed, the copy of
$i-1$ in row $R$ of $M(T)$ which was an $i$ in $T$ cannot lie directly above an $i$ in 
$M(T)$, as $T$ has strictly increasing columns.  Moreover, there are oddly many copies of
$i$ in row $R$ of $T$ that do not have an $i-1$ immediately above them.  One of these must
be the copy that was decremented to obtain $M(T)$, as this is the leftmost copy.  Thus, there 
are evenly many such copies of $i$ in row $R$ in $M(T)$.  Given that we obtain $M(T)$ without
making any changes to the first $R-1$ rows of $T$ or to any entry smaller than $i$ in row $R$ 
of $T$, it follows that $M(M(T))$ is obtained from $M(T)$ be incrementing the same entry of
$M(T)$ that was decremented in $T$ to obtain $M(T)$.  Thus, $M(M(T)) = T $ in this case.  A 
similar argument handles the case where $M(T)$ is obtained from $T$ by instead 
incrementing an entry.

\begin{example} 
\label{SSYT-exam}
If $r=c=t=2$, the matching $M$ is
$$
\begin{matrix} 0&2\\3&3
\end{matrix}
\leftrightarrow
\begin{matrix} 1&2\\3&3
\end{matrix},
\qquad 
\begin{matrix} 0&2\\2&3
\end{matrix}
\leftrightarrow
\begin{matrix} 1&2\\2&3
\end{matrix},
\qquad 
\begin{matrix} 1&1\\2&2
\end{matrix}
\leftrightarrow
\begin{matrix} 1&1\\2&3
\end{matrix},
\qquad
\begin{matrix} 0&0\\3&3
\end{matrix}
\leftrightarrow
\begin{matrix} 0&1\\3&3
\end{matrix}
$$
$$
\begin{matrix} 0&0\\2&3
\end{matrix}
\leftrightarrow
\begin{matrix} 0&1\\2&3
\end{matrix},
\qquad 
\begin{matrix} 0&0\\1&3
\end{matrix}
\leftrightarrow
\begin{matrix} 0&1\\1&3
\end{matrix},
\qquad 
\begin{matrix} 0&0\\2&2
\end{matrix}
\leftrightarrow
\begin{matrix} 0&1\\2&2
\end{matrix},
\qquad
\begin{matrix} 0&0\\1&2
\end{matrix}
\leftrightarrow
\begin{matrix} 0&1\\1&2
\end{matrix}
$$
with unmatched points
$$
\begin{matrix} 2&2\\3&3
\end{matrix},
\qquad
\begin{matrix} 1&1\\3&3
\end{matrix},
\qquad
\begin{matrix} 0&2\\1&3
\end{matrix},
\qquad
\begin{matrix} 0&0\\1&1
\end{matrix}.
$$

\end{example}

\begin{proposition}
$M$ is a Morse matching.
\end{proposition}

\begin{proof} 
Our proof that $M$ is acyclic is similar to the one 
given in the
previous section.  Assume for sake of contradiction 
that $C$ is a directed cycle in the directed graph $D(P,M)$ obtained from $M$.  Find $T$ and 
$M(T)$ in $C$ such that $T\leq M(T)$, chosen so that the row $R$ in which $T$ and $M(T)$
differ is as small as possible within the cycle $C$, and among such semistandard Young 
tableaux $S$ in $C$  such that $M(S) $ is obtained from $S$ by incrementing an entry in
row $R$, the entry $i$ of $R$ which is incremented to obtain $M(T)$ from $T$ is as small
as possible.  Since $C$ is a cycle, at some point in $C$ the resulting entry $i+1$ that is 
present in row $R$ of $M(T)$ but instead has value $i$ in $T$ must be decremented.  Say 
the first time (after obtaining $M(T)$ from $T$ in $C$) that such decrementation occurs, 
we obtain tableau $U$ from tableau $V$.  Then it follows that $V=M(U)$, which contradicts
the cycle $C$ having a directed edge downward from $V$ to $U$.
\end{proof}

Now we describe $P^{\unmatched }$.

\begin{lemma}
\label{crit-cell}
In any element of $P^{\unmatched }$, every  even value $2i$ with $2i < t+r-1$  has the odd 
value $2i+1$ just below it.  For each odd value $2i+1,$ the number of 
copies of $2i+1$ in a given row not having
$2i$ immediately above them is even.
\end{lemma}

\begin{proof}
Start with the top row, and proceed downward from row $R$ to row $R+1$ 
by induction as follows.  In row 1, notice that each odd value must 
occur with even multiplicity, since otherwise we could match by 
decrementing the leftmost copy of the odd value.  Thus, each even 
value $2i$ in row 1 will have an even number of copies (possibly 0) of 
$2i+1$ just to its right; therefore we could match by increasing the
rightmost copy of $2i$ to $2i+1$ unless there is a $2i+1$ just below 
it.  Now since our fillings are semistandard, we also 
must have $2i+1$ just below all the other copies of $2i$ in that row.  
This handles row 1.

The same argument works at row $R+1$ once the claim has been proven 
through row $R$.  That is, we may have some odd values in row $R+1$ 
which were forced to be there by virtue of even values in row $R$, 
but all remaining spots to be filled in row $R+1$ will have odd values 
just above them, which ensures that each odd value $2i+1$ that is put 
in any of these remaining positions must occur with even multiplicity 
(not counting the occurrences of the value already forced by row
$R$) to avoid matching by decrementing $2i+1$.  And again any even
value $2i$ in row $R+1$ will have an even number of decrementable 
copies of $2i+1$ to its right, allowing matching by incrementing the 
rightmost $2i$ unless either
it has a copy of $2i+1$ just below it or it is the 
absolute largest allowable value. 
\end{proof}

One may check that $M$ is a Morse
matching meeting the requirements of Lemma \ref{Morse-lemma}.

\begin{corollary} 
\label{pp-even}
If $T\in P^{\unmatched}$ is a critical cell, then
$rank(T)$ in the poset $SSYT(r,c,t)$ is even.
\end{corollary}

\begin{proof} We will use \eqref{rank-eqn} to find the parity of $rank(T).$
For any $T\in P^{unM}$, the first row sum 
of $T$ is even.  If the row length $c$ is odd, then the row sums 
alternate in parity.  Otherwise, all row sums are even.  So the sum of all entries in 
$T$ is even if $c$ is even, and is odd only when $c$ is odd, and 
$r\equiv 2,3 \mod 4.$  However, $c\binom{r}{2}$ is also odd only when $c$ is odd, and 
$r \equiv 2,3 \mod 4,$ so $rank(T)$ must be even.  
\end{proof}

From Corollary~\ref{pp-even} the elements of $P^{\unmatched}$ are equinumerous with 
self-complementary plane partitions. The next Proposition describes 
the elements of $P^{\unmatched}$ as domino tableaux.

\begin{proposition} The elements of $P^{\unmatched}$ are in bijection with
\begin{enumerate}
\item  the semistandard domino tableaux of rectangular shape $r\times c$ 
filled with odd values between $0$ and $r+t-1$, when $r+t-1$ is odd,
\item the semistandard domino tableaux of rectangular shape $r\times c$ 
filled with odd values between $0$ and $r+t-2$, allowing also monominos in the 
last row with value $r+t-1,$ when $r+t-1$ is even.
\end{enumerate}
\end{proposition}

\begin{proof}
We use Lemma~\ref{crit-cell} 
to show that our description of $P^{\unmatched}$  amounts to saying 
that $r\times c$ is tiled by two types of dominos: 
\begin{enumerate}
\item horizontal dominos in which both entries of the domino have
odd value $2i+1$, and 
\item vertical dominos in which the top entry is $2i$ and the bottom
entry is $2i+1$, along with perhaps some monominos in the bottom row of 
maximal allowed value $r+t-1,$ in the case that $r+t-1$ is even. 
\end{enumerate}

These tilings are transformed into the desired set of semistandard domino tableaux 
by simply replacing each $2i$ in the vertical dominos by $2i+1$.  
To see that we indeed get such a 
tiling, first notice that each non-maximal even value must belong to one of our vertical dominos,
and then notice that the remaining odd values of non-maximal value occur in horizontal pairs,
i.e. horizontal dominos, by our even parity requirement above. 

It remains to consider the maximal value $r+t-1,$ which must be in the last row. 
If $r+t-1$ is odd, the number of $r+t-1$ in the last row which do not have $r+t-2$ 
above them must be even, otherwise 
$r+t-1$ is available to be matched by case (1). Thus we tile this portion of the 
last row with horizontal dominos labelled with the odd number $r+t-1.$ It $r+t-1$ is even, 
there is no restriction on the even entries $r+t-1$ in the last row, place monominos there. 
\end{proof}

Note that if $r$, $c$, and $t$ are all odd, $P^{\unmatched}=\varnothing,$ because 
$r+t-1$ is odd, $rc$ is odd, and no monominos are allowed.
This is expected, because there are no self-complementary plane partitions in this case.

\begin{theorem} 
\label{box-thm}
The complex $(\CC (r,c,t),d)$ has homology concentrated in ranks of even parity. 
Moreover, a basis for the homology is given by Lemma~\ref{Morse-lemma}(iii) and
Lemma~\ref{crit-cell}. 
\end{theorem}

It now follows from 
\cite[Theorem 3.1(d)]{St}  and the comments thereafter that the number of self-complementary
plane partitions in an $r\times c \times t$ box is indeed $X(r,c,t,-1)$, since these are 
equinumerous 
with semistandard domino tableaux of rectangular shape $r \times c$  with appropriate 
corresponding  bound on value.

We next use the Morse matching $M$ to decompose  $SSYT(r,c,t)$ into a disjoint 
union of Boolean algebras, generalizing \cite[Proposition 11]{FRST} for $P(k,\ell).$ 

Let $Bottom(r,c,t)$ be the set of all $T\in SSYT(r,c,t)$  for which no odd entry 
satisfies (1) in Definition~\ref{M-def}. For each $T_{bottom}\in Bottom(r,c,t),$ let 
$T_{top}$ be the SSYT obtained by incrementing all 
possible even values which occur in case (2) in any row. Since these incrementations may be 
done independently, the interval $[T_{bottom}, T_{top}]$ is a Boolean algebra in $SSYT(r,c,t).$

\begin{theorem}
\label{SSYT-Boolean} 
As a set, $SSYT(r,c,t)$ is a disjoint union of Boolean algebras
$$
SSYT(r,c,t)=\cup_{T_{bottom}\in Bottom(r,c,t)} [T_{bottom}, T_{top}].
$$
\end{theorem}

\begin{proof} For any $T\in SSYT(r,c,t)$, consider the set 
$$
D(T)=\{(R,i): {\text{an odd $i$ may be decremented to $i-1$ in row $R$ of $T$ by case (1) of 
\ref{3d-matching-definition}}}\}
$$
of possible decrementations of odd values, and the analogous set of incrementations
$$
E(T)=\{(R,i): {\text{an even $i$ may be incremented to $i+1$ in row $R$ of $T$ by case (2) of
\ref{3d-matching-definition}}}\}.
$$

First we check that $(R,i)\in D(T)$ implies $(R,i-1)\notin E(T):$ in row $R,$ $T$ 
has an odd number of $(i-1)+1=i$ to the right of any $i-1$ which do not have $i-1$ 
immediately above them. Similarly one may check that 
$(R,i)\in E(T)$ implies $(R,i+1)\notin D(T).$

If $i\in T$ has been decremented in row $R$ to obtain $T'$, one may check that
$$
D(T')=D(T)-\{(R,i)\}, \quad E(T')=E(T)\cup \{(R,i-1)\},
$$
and also if $i\in T$ has been incremented in row $R$ to obtain $T'$, one may check that
$$
E(T')=E(T)-\{(R,i)\}, \quad D(T')=D(T)\cup \{(R,i+1)\}.
$$
Thus for a given $T\in SSYT(r,c,t)$, $T_{top}$ is found by incrementing $T$ until 
$E(T')=\varnothing,$ and $T_{bottom}$ is found by decrementing $T$ until 
$D(T')=\varnothing.$ 

The interval $[T_{bottom},T_{top}]$ is a Boolean algebra, because 
$T_{top}$ agrees with $T_{bottom}$, except at $E(T_{bottom})$, where they differ by 1.
\end{proof}

\begin{example} The Boolean algebras in Theorem~\ref{SSYT-Boolean}
for $r=c=t=2$ consist of
 the four unmatched points in Example~\ref{SSYT-exam}, and the intervals
 $$
 \left[\begin{matrix} 0&0\\1&2
\end{matrix},\  
\begin{matrix} 0&1\\1&3
\end{matrix}\right],\quad
\left[\begin{matrix} 0&2\\2&3
\end{matrix},\  
\begin{matrix} 1&2\\3&3
\end{matrix}\right],\quad
\left[\begin{matrix} 1&1\\2&2
\end{matrix}, \ 
\begin{matrix} 1&1\\2&3
\end{matrix}\right],\quad
$$
$$
 \left[\begin{matrix} 0&0\\3&3
\end{matrix}, \ 
\begin{matrix} 0&1\\3&3
\end{matrix}\right],\quad
\left[\begin{matrix} 0&0\\2&2
\end{matrix}, \ 
\begin{matrix} 0&1\\2&3
\end{matrix}\right],\quad
 $$
 \end{example}
Theorem~\ref{SSYT-Boolean} immediately yields the following enumerative result regarding
rank-generating functions of posets by virtue of our decomposition into Boolean algebras and
the combinatorial definition of a Schur function as a weighted sum over semistandard Young tableaux.

\begin{corollary} Let $\theta$ be the rectangular partition $c^r.$ 
Then the principally specialized Schur function 
has the expansion
$$
\begin{aligned}
s_{\theta}(1,q,\dots, q^{r+t-1})=q^{c\binom{r}{2}}X(r,c,t,q)=\sum_{T\in Bottom(r,c,t)} q^{||T||} (1+q)^{|E(T)|}.
\end{aligned}
$$
\end{corollary}

Again the Boolean algebras of rank 0 consist of our $M$-unmatched elements,
$T=T_{top}=T_{bottom}.$

\section{A counterexample and a conjecture}
\label{not-even}


Here we show that, despite the positive indications obtained from Theorem  \ref{odd-acyclic} and Theorem \ref{rectangle-thm}, there are infinitely many permutation groups $G$ such that the homology of $(\CC(\FF_2)^G,D)$ is not concentrated in even dimensions.  On the other hand, we conjecture that if $G \leq S_n$ is generated by an $n$-cycle then such homology concentration in even dimensions exists.  We provide evidence for this conjecture, along with a stronger conjecture giving a formula for the ranks of all homology groups under consideration.

Let $G \leq \Sym_n$.  By Theorem \ref{deBruijn-theorem} and Proposition \ref{four-complexes}(2), if $G$ has no self-complementary orbit on ${{[n]} \choose {n/2}}$ then the Euler characteristic of $(\CC(\FF_2)^G,D)$ is zero.  On the other hand, if $n$ is even and $G$ is transitive, then $H_0((\CC(\FF_2)^G,D))$ is nontrivial, by Proposition \ref{odd-acyclic}.  The next result follows.

\begin{proposition} \label{nocon}
Let $n$ be an even integer.  If $G \leq \Sym_n$ is transitive but has no self-complementary orbit on ${{[n]} \choose {n/2}}$ then the homology of $(\CC(\FF_2)^G,D)$ is not concentrated in even dimensions.
\end{proposition}

As shown by J. R. Isbell in \cite{Is1}, $G \leq \Sym_n$ has a self-complementary orbit on ${{[n]} \choose {n/2}}$ if and only if $G$ contains a derangement whose order is a power of two.  (This is not hard to prove.)  Isbell shows also in \cite{Is1} that if $n=2b$ with $b>1$ odd, then there exists a transitive subgroup of $\Sym_n$ containing no derangement whose order is a power of two.  As we found copies of \cite{Is1} to be not easily available, we mention here a construction of such subgroups.  Given an odd integer $b>1$, find the smallest integer $d$ such that $b$ divides $2^d-1$.  Let $V=\FF_2^d$.  There exists a cyclic subgroup $C \leq GL(V)$ such that $C$ has order $b$ and acts irreducibly on $V$.  Form the semidirect product $G:=CV$.  Choose an ordered basis $e_1,\ldots,e_{d-1},e_d$ for $V$ such that the matrix of a generator $c$ for $C$ is in rational canonical form with respect to this basis.  So, $e_ic=e_{i+1}$ for $1 \leq i \leq d-1$ and $e_dc=\sum_{i=1}^d\alpha_ie_i$, where the characteristic polynomial of $c$ is $x^d+\sum_{i=1}^d\alpha_ix^{i-1}$.   As $c$ acts irreducibly on $V$, $1$ is not a root of this characteristic polynomial.   It follows that $\sum_{i=1}^d \alpha_i \equiv 0$  Set
\[
H:=\left\{\sum_{i=1}^da_ie_i \in V:\sum_{i=1}^da_i \equiv 0\right\}.
\]
Now $H$ is a hyperplane in $V$ and therefore a subgroup of $G$.  Note that $[G:H]=2b$.  As $C$ acts irreducibly on $V$, $H$ contains no nontrivial normal subgroup of $G$.  Thus the action of $G$ on the cosets of $H$ by translation gives an embedding of $G$ in $\Sym_{2b}$.  The conjugacy classes of nontrivial elements of $2$-power order in $G$ are exactly the orbits of $C$ on $V$.  Therefore, to show that no $2$-element of $G$ is a derangement, it suffices to show that every such orbit intersects $H$ nontrivially.  Pick some nonempty $I \subseteq [d]$ and let $v=\sum_{i \in I}e_i \in V$.  If $|I|$ is even then $v \in H$.  If $|I|$ is odd, let $m$ be the largest element of $I$.  Then
\[
vc^{d-m+1}=e_dc+\sum_{j \in I \setminus \{m\}}e_{j+d-m} \in H.
\]
We have the following conclusion.

\begin{proposition}
If $b>1$ is odd, there exists some $G \leq \Sym_{2b}$ such that the homology of $(\CC(\FF_2)^G,D)$ is not concentrated in even dimensions.
\end{proposition}

Before turning to $n$-cycles, we remark that Isbell conjectured that if we fix an odd number $b$, then for all large enough $a$, every transitive subgroup of $\Sym_{2^ab}$ contains a derangement whose order is a power of two.

Although the homology of the complexes  introduced in Section ~\ref{chain-complex-section}
is not always concentrated in even dimensions, 
we are still interested in interesting families of groups $G$
for which this concentration does occur.
In this situation, the Poincar\'e polynomial for the homology of,
say, $(\CC(\FF_2)^G,D)$, can be interpreted as giving a grading on the set
of self-complementary $G$-orbits.

We conclude with evidence that this concentration holds for $G$ a cyclic group $C_n$
generated by an $n$-cycle in
$\Sym_n$.  In this case the orbits are  necklaces, as in  \cite{Re}. 
Here are  homology calculations from {\tt
Mathematica} on $C_n$ for $n$ even; the case where $n$ is odd is settled by 
Proposition \ref{odd-acyclic}.

$$
\begin{matrix}
n &\text{ homology ranks } \\
 &         \\
2& {1, 0, 0} \\
4& {1, 0, 1, 0, 0}\\
6&{1, 0, 0, 0, 1, 0, 0}\\
8&{1, 0, 1, 0, 1, 0, 1, 0, 0}\\
10& {1, 0, 0, 0, 2, 0, 0, 0, 1, 0, 0}\\
12&{1, 0, 1, 0, 2, 0, 2, 0, 1, 0, 1, 0, 0}\\
14&{1, 0, 0, 0, 3, 0, 2, 0, 3, 0, 0, 0, 1, 0, 0}\\
16&{1, 0, 1, 0, 3, 0, 5, 0, 5, 0, 3, 0, 1, 0, 1, 0, 0}\\
18&{1, 0, 0, 0, 4, 0, 6, 0, 8, 0, 6, 0, 4, 0, 0, 0, 1, 0, 0}
\end{matrix}
$$

The following calculation of $X(C_n,q)$ and  $X(C_n,-1)$ may be done using 
Polya-Redfield theory or Burnside's lemma.

\begin{proposition}
\label{cyclic-Polya-calculation}
For $n>2$, 
$$
\begin{aligned}
X(C_n,q) &= \frac{1}{n} \sum_{d: d | n} \varphi(d)  (1+q^d)^{\frac{n}{d}},\mbox{ and}\\
X(C_n,-1) &= \frac{1}{n} \sum_{\substack{d: d | n,\\ d\text{ even
}}} \varphi(d) 2^{\frac{n}{d}},
\end{aligned}
$$ where $\phi (d)$ denotes Euler's totient function.
\end{proposition}

Note that the above homology data suggest that 
\begin{enumerate}
\item[$\bullet$]
$H_i(\CC^{C_{2n}})=0$ for $i$ odd, and for $i=2n-1,2n$, and
\item[$\bullet$]
$H_{2j}(\CC^{C_{2n}})=H_{2n-2-2j}(\CC^{C_{2n}})$.
\end{enumerate}

But one can be much more precise.  Note
that using the formula in Proposition~\ref{cyclic-Polya-calculation} 
for $X(C_n,-1)$, one can easily check that it satisfies
the recursion
$$
\begin{aligned}
X(C_{2n},-1)
& = \frac{X(C_n,-1)+X(C_n,1)}{2}. \\
\end{aligned}
$$

\noindent
This recursion may be modified to  a recursion
predicting the homology Poincar\'e series.  
For notational convenience, define
$$
\begin{aligned}
A_n(q) &:=\sum_{i=1}^{n}\dim H_i((\CC(\FF_2)^{C_n},D))q^i \\
X_n(q) &:=X(C_n,q).
\end{aligned}
$$

\begin{conjecture}
\label{c2nconj}
For any positive integer $n$, one has $H_i((\CC^{C_{2n}},D))=0$ if $i$ is odd, so that
$A_{2n}(q)$ is a polynomial in $q^2$, and one has the recursion
$$
\begin{aligned}
A_{2n}(q^{1/2}) 
& = \frac{qA_n(q)+X_n(q)}{1+q}. \\
\end{aligned}
$$
\end{conjecture}

We have some strong evidence for Conjecture \ref{c2nconj}.  It is correct at $q=1$.
It holds through $n=18$, as shown earlier. 
We have proven it for $n$ odd, although we do 
not include a proof here.  In addition, it is correct with regard to its prediction 
about $H_2((\CC^{C_{2n}},D))$.

\section{Acknowledgments}

The authors are grateful to Vic Reiner for his numerous helpful suggestions both 
regarding mathematics and also his substantial assistance with the exposition.

\end{document}